\UseRawInputEncoding
\documentclass[letterpaper,10pt]{amsart}
\usepackage[all]{xy}                        %

\CompileMatrices                            

\UseTips                                    

\input xypic
\usepackage[bookmarks=true]{hyperref}       

\usepackage{amssymb,latexsym,amsmath,amscd}
\usepackage{xspace}
\usepackage{color}
\usepackage{kpfonts}
\usepackage{graphicx}
\usepackage{multicol}
\reversemarginpar

\theoremstyle{plain}
\newtheorem{theorem}{Theorem}[section]
\newtheorem*{theorem*}{Theorem}
\newtheorem{proposition}[theorem]{Proposition}

\newtheorem{lemma}[theorem]{Lemma}

\theoremstyle{definition}
\newtheorem{definition}[theorem]{Definition}

\newtheorem{remark}[theorem]{Remark}
\newtheorem{example}[theorem]{Example}

\newcommand{\enm}[1]{\ensuremath{#1}}          %
\newcommand{\op}[1]{\operatorname{#1}}
\newcommand{\cal}[1]{\mathcal{#1}}

\newcommand{\ZZ}{\enm{\mathbb{Z}}}

\newcommand{\PP}{\enm{\mathbb{P}}}

\newcommand{\Bb}{\enm{\cal{B}}}

\newcommand{\Ee}{\enm{\cal{E}}}
\newcommand{\Ff}{\enm{\cal{F}}}
\newcommand{\Gg}{\enm{\cal{G}}}

\newcommand{\Oo}{\enm{\cal{O}}}

\newcommand{\Vv}{\enm{\cal{V}}}

\renewcommand{\phi}{\varphi}
\renewcommand{\theta}{\vartheta}
\renewcommand{\epsilon}{\varepsilon}

\newcommand{\Pic}{\op{Pic}}

\renewcommand{\to}[1][]{\xrightarrow{\ #1\ }}

\newcommand{\old}[1]{}

\begin{document}

\title[Ulrich bundles of arbitrary rank on Segre-Veronese varieties ]{Ulrich bundles of arbitrary rank\\ on Segre-Veronese varieties }

\author{F. Malaspina }

\address{Politecnico di Torino, Corso Duca degli Abruzzi 24, 10129 Torino, Italy}
\email{francesco.malaspina@polito.it}



\begin{abstract}

We generalize the results by Eisenbud and Schreyer about
Ulrich bundles over Veronese varieties to Segre-Veronese varieties. We discuss the range where we have natural cohomology and we construct multigraded resolutions and monads for Ulrich bundles of any rank. Moreover we give cohomological characterizations for significant families of bundles.

\end{abstract}

\maketitle


\section{Introduction}

Let  $\PP^N$ be the $N$-dimensional projective space over an
algebraically closed field $k$ of characteristic $0$. If $X\subseteq\PP^N$ is a smooth closed submanifold of dimension $n \geq 1$ we set $\Oo_X(h):=\Oo_{\PP^N}(1)\otimes\Oo_X$.

Among vector bundles
 $\Vv$ on such an $X$, the simplest ones from the cohomological point of view satisfy the vanishing
$$
H^i\big(X,\Vv(th)\big)=0, \qquad \forall\ i=1,\dots,n-1,\ t\in\mathbb Z.
$$
Such vector bundles are called arithmetically Cohen-Macaulay (aCM for
short).

There has been increasing interest on the classification of  aCM bundles on various projective varieties, which is important in a sense that the aCM bundles are considered to give a measurement of complexity of the underlying space. Moreover the aCM bundles are a crucial ingredient for the study of any bundle on $X$ as it is showed in \cite{MRa}.  A special type of aCM bundles, called the Ulrich bundles, are the ones achieving the maximum possible minimal number of generators given by $\deg(X)rank(\Vv)$. These bundles, originally studied for computing Chow forms,
conjecturally exist over any variety (see \cite{ES}) and they
are characterized by the linearity of the minimal graded free resolution
over the polynomial ring of their module of global sections. Many recent papers are devoted to the study of Ulrich bundles over significant varieties (see \cite
{A, ACM, CCHMW, CM1, CM2}).

In \cite{CMP} the authors give families of arbitrary large dimensions over Segre varieties  (except $\PP^1\times\PP^1$).  In \cite{FM} it is shown that every  Ulrich bundle over a rational normal scroll of dimension two is obtained as an extension between two direct sums of line bundles and, in case of quartic scrolls, all the aCM bundles are classified.  In \cite{AHMP} are classified Ulrich vector bundles of arbitrary rank on smooth projective varieties of minimal degree  of any dimension. They are characterized as the bundles admitting a special type of filtration. A consequence of this result is that the moduli spaces of Ulrich bundles are zero-dimensional. The case $\PP^2\times\PP^1$ is very peculiar, in fact there exists only a finite number of aCM bundles which are not Ulrich;  in \cite{FMS} all the aCM bundles are classified. The case of $\PP^2\times\PP^2$ is much more complicated and has been studied in \cite{Mal}. The bigraded resolutions
of Ulrich bundles of arbitrary rank are constructed. Cohomological characterizations of the Ulrich  bundles $\Vv$  with $h^1(\Vv\otimes\Omega_{\PP^2}\boxtimes\Omega_{\PP^2})=0$, or with $h^2(\Vv\otimes\Omega_{\PP^2}(-1)\boxtimes\Omega_{\PP^2}(-1))=0$, or obtained as pullback from $\PP^2$, are proved. The case of Veronese varieties has been studied in \cite{ES}. The authors show that every Ulrich bundle has natural cohomology and can be obtained from a graded resolution.

  The aim of this article is to generalize the cohomological characterizations of Ulrich bundles
from \cite{ES} and \cite{Mal} to Segre Veronese varieties $\Sigma$ obtained by embedding  $\mathbb P^{n_1}\times\dots\times\mathbb P^{n_s}$ with $\Oo(k_1,\dots ,k_s)$ ($s>0, k_1,\dots, k_s\geq 1)$. More precisely we show that every Ulrich bundle $\Vv$ on $\Sigma$ is regular according to both the two different notions of Castelnuovo-Mumford regularity given in \cite{bm2} and \cite{hw}, moreover it has natural cohomology (as in \cite{ES} on Veronese varieties) in a suitable range. Then we compute the cohomology of $\Vv$ tensored with the pullbacks of the $\Omega^i$ bundles  with suitable twists.
  We introduce the following notation:
$$\alpha_i^{a_1,\dots ,a_s}=h^i(\Vv(-ih)\otimes\Omega_{\PP^{n_{1}}}^{a_1}(a_1)\boxtimes\dots\boxtimes\Omega_{\PP^{n_{s}}}^{a_s}(a_s))$$
where $0\leq a_i\leq n_i$ and we show that $$h^k(\Vv(-ih)\otimes\Omega_{\PP^{n_{1}}}^{a_1}(a_1)\boxtimes\dots\boxtimes\Omega_{\PP^{n_{s}}}^{a_s}(a_s))=0$$ if $k\not=i$.

Thanks to a Beilinson type spectral sequence constructed with suitable exceptional collections we prove that, for $q=0,\dots, n_1+\dots +n_s$,
$\Vv(-qh)$ is the homology of the monad

$$0\to\Bb_1\to\oplus_{a_1+\dots +a_s=q}\Oo(-a_1,\dots ,-a_s)^{\oplus \alpha_{q}^{a_1,\dots ,a_s}}\to\Bb_2\to 0$$
where $\Bb_1$ is defined by the exact sequence
$$0\to\Oo(-n_1,\dots ,-n_s)^{\oplus \alpha_{q}^{n_1,\dots ,n_s}}\to\dots\to \oplus_{a_1+\dots +a_s=q+1}\Oo(-a_1,\dots ,-a_s)^{\oplus \alpha_{q}^{a_1,\dots ,a_s}}\to\Bb_1\to 0$$
and $\Bb_2$ is defined by the exact sequence
$$0\to\Bb_2\to \oplus_{a_1+\dots +a_s=q-1}\Oo(-a_1,\dots ,-a_s)^{\oplus \alpha_{q}^{a_1+\dots +a_s}}\to\dots\to\oplus_{a_1+\dots +a_s=1}\Oo(-a_1,\dots ,-a_s)^{\oplus \alpha_{q}^{a_1+\dots +a_s}}\to 0.$$
Notice that for $q=0, q=1$ and $q=n_1+\dots +n_1$ we obtain a multigraded resolution.

Here we summarize the structure of this article.  In section \ref{sec2}  we make the cohomological computations related to an Ulrich bundle and we discuss the Castelnuovo-Mumford regularity and the range where we have natural cohomology. In section \ref{sec3}  we prove the main results for the case of biprojective spaces. In section \ref{sec4} we give cohomological characterizations for significant families of Ulrich bundles. In section \ref{sec5} we deal with the general case.


\





\section{Cohomology of Ulrich bundles on Segre Veronese varieties}\label{sec2}

 Let $\Sigma=\Sigma_{n_1,\dots ,n_s}^{k_1,\dots, k_s}$ be the Segre-Veronese variety obtained by embedding  $X=\mathbb P^{n_1}\times\dots\times\mathbb P^{n_s}$ with $\Oo(k_1,\dots ,k_s)$ ($s>0, k_1,\dots, k_s\geq 1)$. Let $d=n_1+\dots +n_s$ and $h=k_1h_1+\dots +k_sh_s$. We have $$\deg (\Sigma_{n_1,\dots ,n_s}^{k_1,\dots, k_s})=(k_1h_1+\dots +k_sh_s)^s=k_1^{n_1}\dots k_s^{n_s}\deg(\Sigma_{n_1,\dots ,n_s})=k_1^{n_1}\dots k_s^{n_s}\frac{d!}{\prod_{i=1}^{s}(n_i)!}.$$
Moreover we  have $$\Pic (X)\cong \frac{\ZZ\langle h_1,\dots, h_s\rangle}{h_1^{n_1+1},\dots ,h_s^{n_s+1}}$$ and $$\omega_X \cong \Oo_X(-(n_1+1)h_1-\dots -(n_s+1)h_s).$$ We will simply denote $\Oo_X(a_1h_1+\dots +a_sh_s)$ by $\Oo_X(a_1,\dots,a_s)$.

  Let us denote by  ${\displaystyle e^r_j = \left(
\begin{array}{c}
n_r+1 \\
j
\end{array}\right)}$ for $r=1,\dots ,s$.

  We will often use the following exact sequences obtained as pullbacks of Koszul sequences from $\mathbb P^{n_r}$:
  \begin{equation}\label{a1g}0 \to \Oo(0,\dots,-n_r-1,\dots ,0)\to
 \Oo^{e^r_{n_r}}(0,\dots,-n_r,\dots ,0)
\to \cdots \to  \Oo^{e^r_{1}}(0\dots,-1,\dots ,0) \to \Oo\to 0,\end{equation}
  and for $a_r=1,\dots ,n_r-1$, $l=1,\dots, s$

\begin{equation}\label{a3g}0 \to \Oo(0,\dots, a_r-n_r-1, \dots, 0) \to
 \Oo^{e^r_{n_r}}(0,\dots, a_r-n,\dots, 0)
\to \cdots \end{equation}
$$\cdots\to  \Oo^{e^r_{a_r+1}}(0,\dots, -1,\dots ,0) \to \Oo_{\PP^{n_1}}\boxtimes\dots\boxtimes\Omega_{\PP^{n_{i}}}^{a_r}(a_r)\boxtimes\dots\boxtimes\Oo_{\PP^{n_s}}\to 0,$$

\begin{equation}\label{a5g}0 \to \Oo_{\PP^{n_1}}\boxtimes\dots\boxtimes\Omega_{\PP^{n_{r}}}^{a_r}(a_r)\boxtimes\dots\boxtimes\Oo_{\PP^{n_s}} \to
 \Oo^{e^r_{{a_r}}}
\to \cdots\end{equation} $$ \to  \Oo^{e^r_{1}}(0,\dots , a_r-1,\dots, 0) \to \Oo(0,\dots, a_r, \dots, 0)\to 0,$$

\begin{equation}\label{a7g}0 \to \Omega_{\PP^{n_{1}}}^{a_1}(a_1)\boxtimes\dots\boxtimes\Omega_{\PP^{n_{l-1}}}^{a_{l-1}}(a_{l-1})\boxtimes\Oo(a_l-n_l-1, 0, \dots, 0) \to
 \Omega_{\PP^{n_{1}}}^{a_1}(a_1)\boxtimes\dots\boxtimes\Omega_{\PP^{n_{l-1}}}^{a_{l-1}}(a_{l-1})\boxtimes\Oo^{e^l_{n_l}}(a_l-n,0,\dots, 0)
\to \cdots \end{equation}
$$\cdots\to  \Omega_{\PP^{n_{1}}}^{a_1}(a_1)\boxtimes\dots\boxtimes\Omega_{\PP^{n_{l-1}}}^{a_{l-1}}(a_{l-1})\boxtimes\Oo^{e^l_{a_l+1}}(-1,0,\dots ,0) \to \Omega_{\PP^{n_{1}}}^{a_1}(a_1)\boxtimes\dots\boxtimes\Omega_{\PP^{n_{l}}}^{a_{l}}(a_{l})\boxtimes\Oo\to 0,$$

\begin{equation}\label{a9g}0 \to \Omega_{\PP^{n_{1}}}^{a_1}(a_1)\boxtimes\dots\boxtimes\Omega_{\PP^{n_{l}}}^{a_{l}}(a_{l})\boxtimes\Oo \to
 \Omega_{\PP^{n_{1}}}^{a_1}(a_1)\boxtimes\dots\boxtimes\Omega_{\PP^{n_{l-1}}}^{a_{l-1}}(a_{l-1})\boxtimes\Oo^{e^l_{{a_l}}}
\to \cdots\end{equation} $$ \to  \Omega_{\PP^{n_{1}}}^{a_1}(a_1)\boxtimes\dots\boxtimes\Omega_{\PP^{n_{l-1}}}^{a_{l-1}}(a_{l-1})\boxtimes\Oo^{e^l_{1}}(a_l-1,0,\dots, 0) \to \Omega_{\PP^{n_{1}}}^{a_1}(a_1)\boxtimes\dots\boxtimes\Omega_{\PP^{n_{l-1}}}^{a_{l-1}}(a_{l-1})\boxtimes\Oo(a_r,0, \dots, 0)\to 0.$$\\

We will use the following definition of Ulrich bundles (see \cite{ES}):
\begin{definition} A bundle $\Vv$  on $\Sigma$ is Ulrich if $H^i(\Vv(jh))=0$ for any $i$ and $j=-1, \dots ,-d$
\end{definition}

\begin{proposition}
\label{riv}
Let $\Vv$ be an Ulrich bundle on $\Sigma$.
\begin{enumerate}

\item[(i)] For any $i>0$, $H^i(\Vv(-ih)\otimes\Oo(j_1,\dots, j_s))=0$ if $j_1,\dots, j_s\geq 0$.
\item[(ii)] For any $i<d$, $H^i(\Vv(-(i+1)h)\otimes\Oo(j_1,\dots, j_s))=0$ if $j_1,\dots, j_s\leq 0$.
\end{enumerate}
\end{proposition}
\begin{proof}
$H^i(\Vv(jh))=0$ for any $i$ and $j=-1, \dots ,-d$ since $\Vv$ is Ulrich.

Since $H^d(\Vv(-dh))=0$,  $H^d(\Vv(-dh)\otimes\Oo(j_1,\dots, j_s))=0$ if $j_1,\dots, j_s\geq 0$.\\

Let assume $(i)$ for any index $i\geq i_0>1$. We prove $(a)$ for $i_0-1$.\\

   From (\ref{a1g}) tensored by $\Vv(-(i_0-1)h)\otimes\Oo(0,\dots,1,\dots, 0)$  since for any $k \geq 0$ $$H^{i_0-1+k}(\Vv(-(i_0-1)h)\otimes\Oo(0,\dots,-k,\dots, 0))=H^{i_0-1}(\Vv(-(i_0-1)h))=0,$$ we get $$H^{i_0-1}(\Vv(-(i_0-1)h)\otimes\Oo(0,\dots,1,\dots, 0))=0.$$

   From (\ref{a1g}) tensored by $\Vv(-(i_0-1)h)\otimes\Oo(j_1,\dots,j_r+1,\dots, j_s)$ with $j_1, \dots, j_s\geq 0$  since for any $k \geq 0$ $$H^{i_0-1+k}(\Vv(-(i_0-1)h)\otimes\Oo(j_1,\dots,j_r-k,\dots, j_s))=0,$$ we get $$H^{i_0-1}(H^i(\Vv(-(i_0-1)h)\otimes\Oo(j_1,\dots, j_s))=0$$ if $j_1,\dots, j_s\geq 0$. So $(i)$ is proved.\\

Since $H^0(\Vv(-h))=0$,  $H^0(\Vv(-h)\otimes\Oo(j_1,\dots, j_s))=0$ if $j_1,\dots, j_s\leq 0$.\\

Let assume $(ii)$ for an index $i\leq i_0<d-1$. We prove $(ii)$ for $i_0+1$.\\

 From (\ref{a5g}) tensored by $\Vv(-(i_0+2)h)\otimes\Oo(0,\dots,n_r,\dots, 0)$  since for any $k \geq 0$ $$H^{i_0+1-k}(\Vv(-(i_0+2)h)\otimes\Oo(0,\dots,k,\dots, 0))=H^{i_0+1}(\Vv(-(i_0+2)h)\otimes\Oo)=0,$$ we get $$H^{i_0+1}(\Vv(-(i_0+2)h)\otimes\Oo(0,\dots,-1,\dots, 0))=0.$$

   From (\ref{a5g}) tensored by $\Vv(-(i_0+2)h)\otimes\Oo(j_1,\dots,n_r+j_r,\dots, j_s)$ with $j_1, \dots, j_s\leq 0$, since for any $k \geq 0$ $$H^{i_0+1-k}(\Vv(-(i_0+2)h)\otimes\Oo(j_1,\dots,j_l+k,\dots, j_s))=0,$$ we get $$H^{i_0+1}(\Vv(-(i_0+2)h)\otimes\Oo(j_1,\dots, j_s))=0$$ when $j_1,\dots, j_s\leq 0$ and $(ii)$ is proved.

\end{proof}

\begin{remark}In particular an Ulrich bundle on $\Sigma$ satisfies the two different notions of Castelnuovo-Mumford regularity
given in \cite{bm2} and \cite{hw}.
\end{remark}

\begin{proposition}\label{riv2g}
Let $\Vv$ be an Ulrich bundle on $\Sigma$. Let $a_r, l, r$ be integers $ 0\leq a_r\leq n_r$ and $l,r=1,\dots ,s$.
\begin{enumerate}

\item[(a)] For any $i>0$, $H^i(\Vv(-ih)\otimes\Oo_{\PP^{n_1}}(j_1)\boxtimes\dots\boxtimes\Omega_{\PP^{n_{r}}}^{a_r}(a_r+1+j_r)\dots\boxtimes\Oo_{\PP^{n_s}}(j_s))=0$ if $j_1, \dots, j_s\geq 0$.
\item[(b)] For any $i<d$, $H^i(\Vv(-(i+1)h)\otimes\Oo_{\PP^{n_1}}(j_1)\boxtimes\dots\boxtimes\Omega_{\PP^{n_{r}}}^{a_r}(a_r+j_r)\dots\boxtimes\Oo_{\PP^{n_s}}(j_s))=0$ if $j_1, \dots, j_s\leq 0$.
    \item[(c)] For any $i>0$, $H^i(\Vv(-ih)\otimes\Omega^{a_1}(a_1+1+j_1)\boxtimes\dots\boxtimes\Omega^{a_l}(a_l+1+j_l)\boxtimes\Oo(j_{l+1},\dots , j_s))=0$ if $j_1, \dots, j_s\geq 0$.
    \item[(d)] For any $i<d$, $H^i(\Vv(-(i+1)h)\otimes\Omega^{a_1}(a_1+j_1)\boxtimes\dots\boxtimes\Omega^{a_l}(a_l+j_l)\boxtimes\Oo(j_{l+1},\dots , j_s))=0$ if $j_1, \dots, j_s\leq 0$.
\end{enumerate}
\end{proposition}
\begin{proof}

$(a)$ Since $H^d(\Vv(-dh))=0$,  from (\ref{a3g}) tensored by $\Vv(-dh)\otimes\Oo(0,\dots,1,\dots ,0)$  we get $$H^d(\Vv(-dh)\otimes\Oo_{\PP^{n_1}}\boxtimes\dots\boxtimes\Omega_{\PP^{n_{i}}}^{a_r}(a_r+1)\boxtimes\dots\boxtimes\Oo_{\PP^{n_s}})=0.$$   By a recursive argument, since $$H^d(\Vv(-dh)\otimes\Oo(j_1,\dots, j_s))=0$$ for $j_1,\dots, j_s\geq 0$,  we get $$H^d(\Vv(-dh)\otimes\Oo_{\PP^{n_1}}(j_1)\boxtimes\dots\boxtimes\Omega_{\PP^{n_{r}}}^{a_r}(a_r+1+j_r)\boxtimes\dots\boxtimes\Oo_{\PP^{n_s}}(j_s))=0$$ when $j_1, \dots, j_s\geq 0$.\\

   Let assume $(a)$ for any index $i\geq i_0>1$. We prove $(a)$ for $i_0-1$.\\

   From (\ref{a3g}) tensored by $\Vv(-(i_0-1)h)\otimes\Oo(0,\dots,1,\dots, 0)$,  since for any $k \geq 0$ $$H^{i_0-1+k}(\Vv(-(i_0-1)h)\otimes\Oo(0,\dots,-k,\dots, 0))=H^{i_0-1}(\Vv(-(i_0-1)h))=0,$$ we get $$H^{i_0-1}(\Vv(-(i_0-1)h)\otimes\Oo_{\PP^m}\boxtimes\dots\boxtimes\Omega^{a_r}(a_r+1)\boxtimes\dots\boxtimes\Oo_{\PP^m})=0.$$

   From (\ref{a3g}) tensored by $\Vv(-(i_0-1)h)\otimes\Oo(j_1,\dots,j_l+1,\dots, j_s)$ with $j_1, \dots, j_s\geq 0$,  since for any $k \geq 0$ $$H^{i_0-1+k}(\Vv(-(i_0-1)h)\otimes\Oo(j_1,\dots,j_l-k,\dots, j_s))=0,$$ we get $$H^{i_0-1}(\Vv(-(i_0-1)h)\otimes\Oo_{\PP^m}(j_1)\boxtimes\dots\boxtimes\Omega^{a_r}(a_r+1-j_l)\boxtimes\dots\boxtimes\Oo_{\PP^m}(j_s))=0.$$

$(b)$ Since $H^0(\Vv(-h))=0$,  from (\ref{a5g})  tensored by
 $\Vv(-h)$  we get $$H^0(\Vv(-h)\otimes\Oo_{\PP^{n_1}}\boxtimes\dots\boxtimes\Omega_{\PP^{n_{i}}}^{a_r}(a_r)\boxtimes\dots\boxtimes\Oo_{\PP^{n_s}})=0.$$   By a recursive argument since $$H^d(\Vv(-h)\otimes\Oo(j_1,\dots, j_s))=0$$ for $j_1,\dots, j_s\leq 0$  we get $$H^0(\Vv(-dh)\otimes\Oo_{\PP^{n_1}}(j_1)\boxtimes\dots\boxtimes\Omega_{\PP^{n_{r}}}^{a_r}(a_r+j_r)\boxtimes\dots\boxtimes\Oo_{\PP^{n_s}}(j_s))=0$$ when $j_1, \dots, j_s\leq 0$.\\

Let assume $(b)$ for an index $i\leq i_0<d-1$. We prove $(b)$ for $i_0+1$.\\

 From (\ref{a5g}) tensored by $\Vv(-(i_0+2)h)$  since for any $k \geq 0$ $$H^{i_0+1-k}(\Vv(-(i_0+2)h)\otimes\Oo(0,\dots,k,\dots, 0))=H^{i_0+1}(\Vv(-(i_0+2)h))=0,$$ we get $$H^{i_0+1}(\Vv(-(i_0+2)h)\otimes\Oo_{\PP^m}\boxtimes\dots\boxtimes\Omega^{a_r}(a_r)\boxtimes\dots\boxtimes\Oo_{\PP^m})=0.$$

   From (\ref{a5g}) tensored by $\Vv(-(i_0+2)h)\otimes\Oo(j_1,\dots,j_l,\dots, j_s)$ with $j_1, \dots, j_s\leq 0$  since for any $k \geq 0$ $$H^{i_0+1-k}(\Vv(-(i_0+2)h)\otimes\Oo(j_1,\dots,j_l+k,\dots, j_s))=0,$$ we get $$H^{i_0+1}(\Vv(-(i_0+2)h)\otimes\Oo_{\PP^m}(j_1)\boxtimes\dots\boxtimes\Omega^{a_r}(a_r+j_l)\boxtimes\dots\boxtimes\Oo_{\PP^m}(j_s))=0.$$

$(c)$ The case $l=1$ is proved in $(a)$. Let assume $(c)$ for $l-1$ with  $2\leq l\leq s$. We prove $(c)$ for $l$.\\

   From (\ref{a7g}) tensored by $\Vv(-ih)\otimes\Oo(1,\dots ,1,1,0,\dots ,0)$  since for any $k \geq 0$ $$H^{i+k}(\Vv(-ih)\otimes\Omega_{\PP^{n_{1}}}^{a_1}(a_1+1)\boxtimes\dots\boxtimes\Omega_{\PP^{n_{l-1}}}^{a_{l-1}}(a_{l-1}+1)\boxtimes\Oo(-k,0,\dots , 0))=0,$$ we get $$H^{i}(\Vv(-ih)\otimes\Omega^{a_1}(a_1+1)\boxtimes\dots\boxtimes\Omega^{a_l}(a_l+1)\boxtimes\Oo)=0.$$

   From (\ref{a7g}) tensored by $\Vv(-ih)\otimes\Oo(j_1+1,\dots,j_l+1,j_{l+1}\dots, j_s)$ with $j_1, \dots, j_s\geq 0$  since for any $k \geq 0$ $$H^{i+k}(\Vv(-ih)\otimes\Omega_{\PP^{n_{1}}}^{a_1}(a_1+j_1+1)\boxtimes\dots\boxtimes\Omega_{\PP^{n_{l-1}}}^{a_{l-1}}(a_{l-1}+j_{l-1}+1)\boxtimes\Oo(j_l-k,j_{l+1},\dots ,j_s))=0,$$ we get $$H^i(\Vv(-ih)\otimes\Omega^{a_1}(a_1+1+j_1)\boxtimes\dots\boxtimes\Omega^{a_l}(a_l+1+j_l)\boxtimes\Oo(j_{l+1},\dots ,j_s))=0.$$

$(d)$ The case $l=1$ is proved in $(b)$. Let assume $(d)$ for $l-1$ with  $2\leq l\leq s$. We prove $(d)$ for $l$.\\

   From (\ref{a9g}) tensored by $\Vv(-(i+1)h)$  since for any $k \geq 0$ $$H^{i-k}(\Vv(-(i+1)h)\otimes\Omega_{\PP^{n_{1}}}^{a_1}(a_1)\boxtimes\dots\boxtimes\Omega_{\PP^{n_{l-1}}}^{a_{l-1}}(a_{l-1})\boxtimes\Oo(-k,0,\dots , 0))=0,$$ we get $$H^{i}(\Vv(-(i+1)h)\otimes\Omega^{a_1}(a_1)\boxtimes\dots\boxtimes\Omega^{a_l}(a_l)\boxtimes\Oo))=0.$$

   From (\ref{a9g}) tensored by $\Vv(-(i+1)h)\otimes\Oo(j_1,\dots,j_l,\dots, j_s)$ with $j_1, \dots, j_s\leq 0$ $$H^{i-k}(\Vv(-(i+1)h)\otimes\Omega_{\PP^{n_{1}}}^{a_1}(a_1+j_1)\boxtimes\dots\boxtimes\Omega_{\PP^{n_{l-1}}}^{a_{l-1}}(a_{l-1}+j_{l-1})\boxtimes\Oo(j_l-k,j_{l+1},\dots , j_s))=0,$$ we get $$H^i(\Vv(-ih)\otimes\Omega^{a_1}(a_1+j_1)\boxtimes\dots\boxtimes\Omega^{a_l}(a_l+j_l)\boxtimes\Oo(j_{l+1},\dots ,j_s))=0.$$

\end{proof}

\begin{remark}\label{perm} The above Proposition holds also up to a permutation of the factors $\mathbb P^{n_1},\dots,\mathbb P^{n_s}$.
\end{remark}

\begin{definition} Let $\Vv$ be an Ulrich bundle on $\Sigma$. Let $a_i=0,\dots, n_i$.  We introduce the following notation:

\begin{equation}\label{alphg}\alpha_i^{a_1,\dots , a_s}=h^i(\Vv(-ih)\otimes\Omega_{\PP^{n_{1}}}^{a_1}(a_i)\boxtimes\dots\boxtimes\Omega_{\PP^{n_{s}}}^{a_s}(a_s))\end{equation}
Notice that $\Omega_{\PP^{n}}^0=\Oo_{\PP^{n}}$ and $\Omega^n_{\PP^{n}}(n)=\Oo_{\PP^{n}}(-1)$
\end{definition}

\begin{remark}\label{suntog} Let $\Vv$ be an Ulrich bundle on $\Sigma$. Let $a_i=0,\dots, n_i$ and $i=0,\dots , d$.  By Proposition \ref{riv2g} we get

$$h^k(\Vv(-ih)\otimes\Omega^{a_1}(a_i)\boxtimes\dots\boxtimes\Omega^{a_s}(a_s))=0$$ for any $k\not= i$.
\end{remark}

\section{Resolutions and Monads for the case $s=2$}\label{sec3}

Let $X=\PP^m\times\PP^n$, let assume $m\geq n$, we  have $\Pic (X)\cong \frac{\ZZ\langle h_1,h_2\rangle}{h_1^{m+1}h_2^{n+1}}$ and $\omega_X \cong \Oo_X(-(n+1)h_1-(m+1)h_2)$. We will simply denote $\Oo_\Sigma(ah_1+bh_2)$ by $\Oo(a,b)$ and $\Omega_{\PP^m}^a(a)\boxtimes\Oo_{\PP^n}, \Oo_{\PP^m}\boxtimes\Omega_{\PP^n}^b(b)$ by $\Omega^a(a)\boxtimes\Oo, \Oo\boxtimes\Omega^b(b)$.\\
  Let $\Sigma=\Sigma_{m,n}^{k, l}$ be the Segre-Veronese variety obtained by embedding  $\mathbb P^{m}\times\mathbb P^{n}$ with $\Oo(k_1,k_2)$ $(k_1,k_2\geq 1)$.  Let $d=m+n$ and $h=kh_1+lh_2$.  $$\deg (\Sigma_{m,n}^{k_1, k_2})=(k_1h_1+ k_2h_2)^d=k^{m}l^{n}\deg(\Sigma_{m,n})=k_1^{m}k_2^{n}\frac{d!}{mn!}.$$

\begin{definition} Let $\Vv$ be an Ulrich bundle on $\Sigma$. Let $a=0,\dots, m$, $b=0,\dots , n$ and $i=0,\dots , d$.  We introduce the following notation:

\begin{equation}\label{alph}\alpha_i^{a,b}=h^i(\Vv(-ih)\otimes\Omega^a(a)\boxtimes\Omega^b(b))\end{equation}
\end{definition}

\begin{remark}\label{sunto} Let $\Vv$ be an Ulrich bundle on $\Sigma$. Let $a=0,\dots, m$, $b=0,\dots , n$ and $i=0,\dots , d$.  By Proposition \ref{riv2g} we get

$$h^k(\Vv(-ih)\otimes\Omega^a(a)\boxtimes\Omega^b(b))=0$$ for any $k\not= i$.
\end{remark}

We will use the following version of Beilinson Theorem (see \cite{AHMP} and  \cite{RU}, \cite{GO}, \cite{BO}):

\begin{theorem}\label{use}
Let $X$ be a smooth projective variety with a full exceptional collection
$\langle E_0, \ldots, E_n\rangle$
where $E_i=\mathcal E_i^*[-k_i]$ with each $\mathcal E_i$ a vector bundle and $(k_0, \ldots, k_n)\in \ZZ^{\oplus n+1}$ such that there exists a sequence $\langle F_n=\mathcal F_n, \ldots, F_0=\mathcal F_0\rangle$ of vector bundles satisfying
\begin{equation}\label{order}
\mathrm{Ext}^k(E_i,F_j)=H^{k+k_i}( \mathcal E_i\otimes \mathcal F_j) =  \left\{
\begin{array}{cc}
\mathbb C & \textrm{\quad if $i=j=k$} \\
0 & \textrm{\quad otherwise}
\end{array}
\right.
\end{equation}
i.e. the collection $\langle F_n, \ldots, F_0\rangle$ labelled in the reverse order is the right dual collection of $\langle E_0, \ldots, E_n\rangle$.
Then for any coherent sheaf $A$ on $X$ there is a spectral sequence in the square $-n\leq p\leq 0$, $0\leq q\leq n$  with the $E_1$-term
\[
E_1^{p,q} = \mathrm{Ext}^{q}(E_{-p},A) \otimes F_{-p}=
H^{q+k_{-p}}(\mathcal E_{-p}\otimes A) \otimes \mathcal F_{-p}
\]
which is functorial in $A$ and converges to
\begin{equation}
E_{\infty}^{p,q}= \left\{
\begin{array}{cc}
A & \textrm{\quad if $p+q=0$} \\
0 & \textrm{\quad otherwise.}
\end{array}
\right.
\end{equation}
\end{theorem}

Now we construct the full exceptional collections that we will use in the next theorems:
 Let us consider on  $\PP^m$ the full exceptional collection $\{\Oo_{\PP^m}(-m),\dots , \Oo_{\PP^m}(-1), \Oo_{\PP^m}\}$ and on  $\PP^n$ the full exceptional collection $\{\Oo_{\PP^n}(-n),\dots , \Oo_{\PP^n}(-1), \Oo_{\PP^n}\}$. We may obtain (see \cite{Orlov}):

\begin{equation}\label{col}\begin{aligned}
&\Ee_{(m+1)(n+1)-1}[k_{(m+1)(n+1)-1}]=\Oo(-m,-n)[-(m+1)(n+1)+d+1]~,~\\
 &\Ee_{(m+1)(n+1)-2}[k_{(m+1)(n+1)-2}]=\Oo(-m+1,-n)[-(m+1)(n+1)+d+1]~,~\\
 & \Ee_{(m+1)(n+1)-3}[k_{(m+1)(n+1)-3}]=\Oo(-m,-n+1)[-(m+1)(n+1)+d+2],\dots \\
& \dots, \Ee_9[k_9]=\Oo(0,-3)[-6]~, \Ee_8[k_8]=\Oo(-1,-2)[-5]~,~ \Ee_7[k_7]=\Oo(-2,-1)[-4]~,~ \Ee_6[k_6]=\Oo(-3,0)[-3], \\
& \Ee_5[k_5]=\Oo(0,-2)[-3]~,~ \Ee_4[k_4]=\Oo(-1,-1)[-2]~,~ \Ee_3[k_3]=\Oo(-2,0)[-1], \\
& \Ee_2[k_2]=\Oo(0,-1)[-1]~,~ \Ee_1[k_1]=\Oo(-1,0)~,~ \Ee_0[k_0]=\Oo.
\end{aligned}
\end{equation}
The associated full exceptional collection $\langle F_d=\mathcal F_d, \ldots, F_0=\mathcal F_0\rangle$ of Theorem \ref{use} is

\begin{equation}\label{cold}\begin{aligned}
&\Ff_{(m+1)(n+1)-1}=\Oo(-1,-1)~,~\Ff_{(m+1)(n+1)-2}[k_{(m+1)(n+1)-2}]=\Omega^{m-1}(m-1)\boxtimes\Oo(-1)~,~\\
 & \Ff_{(m+1)(n+1)-3}=\Oo(-1)\boxtimes\Omega^{n-1}(n-1)), \dots, \Ff_9=\Oo\boxtimes\Omega^3(3)~, \\
&\Ff_8=\Omega^1(1)\boxtimes\Omega^2(2)~,~ \Ff_7=\Omega^2(2)\boxtimes\Omega^1(1)~,~ \Ff_6=\Omega^3(3)\boxtimes\Oo, \\
&\Ff_5=\Oo\boxtimes\Omega^2(2)~,~\Ff_4=\Omega(1)\boxtimes\Omega(1),~ \Ff_3=\Omega^2(2)\boxtimes\Oo~, \\
&\Ff_2=\Oo\boxtimes\Omega(1)~,~ \Ff_1=\Omega(1)\boxtimes\Oo~,~ \Ff_0=\Oo.
\end{aligned}\end{equation}

Let us call $\Gg^{a,b}=\Omega^a(a)\boxtimes\Omega^b(b)$. Notice that $\Omega^m(m)=\Oo(-1)$ and $h^i(\Gg^{a,b}\otimes\Oo(s,t))\not=0$ if and only if $s=a,t=b, i=a+b$. So we get the orthogonality conditions. Notice moreover that $Ext^k(E_i,E_j)=Ext^k(F_i,F_j)=0$ for $k>0$ and any $i,j$,  we have that (\ref{col}) and (\ref{cold}) are  full exceptional collections.

\begin{remark}\label{remb}
It is possible to state a stronger version of the Beilinson's theorem (see  \cite{A} Remark 2.4). Let us consider and let $A$ be a coherent sheaf on $X$. Let $\langle E_0, \ldots, E_n\rangle$ be a full exceptional collection and $\langle F_0, \ldots, F_n\rangle$ its right dual collection. Using the notation of Theorem \ref{use}, if $\langle F_0, \ldots, F_n\rangle$ is strong then there exists a complex of vector bundles $L^\bullet$ such that
\begin{enumerate}
\item $H^k(L^\bullet)=
\begin{cases}
A \ & \text{if $k=0$},\\
0 \ & \text{otherwise}.
\end{cases}$
\item $L^k=\underset{k=p+q}{\bigoplus}H^{q+k_{-p}}(A\otimes E_{-p})\otimes F_{-p}$ with $0\le q \le n$ and $-n\le p \le 0$.
\end{enumerate}
\end{remark}

\begin{theorem}\label{volon}
Let $\Vv$ be an Ulrich bundle on $\Sigma$. Let $a=0,\dots, m$, $b=0,\dots , n$ and $p=0,\dots , d$. Then $\Vv$ arises from an exact sequence of the form:

\begin{equation}\label{res}
0\to\Oo(-m,-n)^{\oplus \alpha_0^{m,n}}\to\Oo(-m,-n+1)^{\oplus \alpha_0^{m,n-1}}\oplus\Oo(-m+1,-n)^{\oplus \alpha_0^{m-1,n}}\to\dots\end{equation}
$$\dots\to\oplus_{a+b=p}\Oo(-a,-b)^{\oplus \alpha_0^{a,b}}\to\dots\to\Oo(-1,0)^{\oplus \alpha_0^{1,0}}\oplus\Oo(0,-1)^{\oplus \alpha_0^{0,1}}\to\Oo^{\oplus \alpha_0^{0,0}}\to \Vv\to 0,$$

or

\begin{equation}\label{res2}
0\to\Oo(-m,-n)^{\oplus \alpha_1^{m,n}}\to\Oo(-m,-n+1)^{\oplus \alpha_1^{m,n-1}}\oplus\Oo(-m+1,-n)^{\oplus \alpha_1^{m-1,n}}\to\dots\end{equation}
$$\dots\to\oplus_{a+b=p}\Oo(-a,-b)^{\oplus \alpha_1^{a,b}}\to\dots\to\Oo(-1,0)^{\oplus \alpha_1^{1,0}}\oplus\Oo(0,-1)^{\oplus \alpha_1^{0,1}}\to \Vv(-h)\to 0,$$

or

\begin{equation}\label{res3}
0\to\Vv(-dh)\to\Oo(-m,-n)^{\oplus \alpha_d^{m,n}}\to\Oo(-m,-n+1)^{\oplus \alpha_d^{m,n-1}}\oplus\Oo(-m+1,-n)^{\oplus \alpha_d^{m-1,n}}\to\dots\end{equation}
$$\dots\to\oplus_{a+b=p}\Oo(-a,-b)^{\oplus \alpha_d^{a,b}}\to\dots\to\Oo(-1,0)^{\oplus \alpha_d^{1,0}}\oplus\Oo(0,-1)^{\oplus \alpha_d^{0,1}}\to  0,$$
or $\Vv(-(d-1)h)$ is the homology of the following monad
\begin{equation}\label{res4}
0\to\Oo(-m,-n)^{\oplus \alpha_{d-1}^{m,n}}\to\Oo(-m,-n+1)^{\oplus \alpha_{d-1}^{m,n-1}}\oplus\Oo(-m+1,-n)^{\oplus \alpha_{d-1}^{m-1,n}}\to\Bb\end{equation} where $\Bb$ is defined by the exact sequence
$$\Bb\to\oplus_{a+b=d-2}\Oo(-a,-b)^{\oplus \alpha_{d-1}^{a,b}}\to\dots\to\Oo(-1,0)^{\oplus \alpha_{d-1}^{1,0}}\oplus\Oo(0,-1)^{\oplus \alpha_{d-1}^{0,1}}\to  0,$$
or for $1<q<d-1$ $\Vv(-qh)$ is the homology of the following monad

\begin{equation}\label{res5}
0\to\Bb_1\to\oplus_{a+b=q}\Oo(-a,-b)^{\oplus \alpha_{q}^{a,b}}\to\Bb_2\to 0,\end{equation}
where $\Bb_1$ is defined by the exact sequence
$$0\to\Oo(-m,-n)^{\oplus \alpha_{q}^{m,n}}\to\dots\to \oplus_{a+b=q+1}\Oo(-a,-b)^{\oplus \alpha_{q}^{a,b}}\to\Bb_1\to 0$$
and $\Bb_2$ is defined by the exact sequence
$$0\to\Bb_2\to \oplus_{a+b=q-1}\Oo(-a,-b)^{\oplus \alpha_{q}^{a,b}}\to\dots\to\Oo(-1,0)^{\oplus \alpha_{q}^{1,0}}\oplus\Oo(0,-1)^{\oplus \alpha_{q}^{0,1}}\to 0.$$

\end{theorem}
\begin{proof}
We consider the Beilinson type spectral sequence associated to $\Vv$ and identify the members of the graded sheaf associated to the induced filtration as the sheaves mentioned in the statement. We consider the full exceptional collection $\Ee_{\bullet}$ given in (\ref{col}) and  collection $\Ff_{\bullet}$ given in (\ref{cold}).

 We construct a Beilinson complex, quasi-isomorphic to $\Vv$, by calculating $H^{i+k_j}(\Vv\otimes \Ff_j)\otimes \Ee_j$ with  $i,j \in \{0, \ldots, d\}$ to get the following table (We put the collection
$\langle E_0, \ldots, E_n\rangle$ in the top row and the collection $\langle F_0, \ldots, F_n\rangle$ in the bottom row, see \cite{AHMP}):

\begin{center}\begin{tabular}{|c|c|c|c|c|c|c|c|c|c|c|}
\hline
$\Oo(-m,-n)$ & $\Oo(-m+1,-n)$	 	 &$\Oo(-m,-n+1)$ &\dots  &$\Oo(0,-n)$ &\dots & $\Oo(-n,0)$& \dots
&$\Oo(0,-1)$&$\Oo(-1,0)$ &$\Oo$\\
\hline
\hline
$H^{d}$        &$H^{d}$    &$0$	 & \dots		 	& $0$  & \dots	    &	 $0$	& \dots		 	& $0$		 	& $0$			& $0$	\\
$H^{d-1}$        &$H^{d-1}$  &$H^{d}$	 & \dots		 	& $0$  & \dots     &	 $0$	 	& \dots		    & $0$		    & $0$			&$0$\\
\vdots     &\vdots     &\vdots	& \vdots		 	& \vdots & \vdots    &	 \vdots	    & \vdots		& \vdots		& 	\vdots		& \vdots	\\
$H^0$        &$H^{0}$  &$H^{1}$	 & \dots		 	& $0$  & \dots     &	 $0$	& \dots		 	& $0$		 	& $0$			& $0$	\\
$0$        &$0$        &$H^0$	 & \dots		 	& $H^d$  & \dots    &	 $0$	 	& \dots	 	    & $0$		 	& $0$		& $0$	\\
\vdots     &\vdots     &\vdots	& \vdots		 	& \vdots & \vdots    &	 \vdots	    & \vdots		& \vdots		& 	\vdots		& \vdots	\\
$0$        &$0$        &$0$	 & \dots		 	& $H^0$  & \dots    &	 $H^d$	 	& \dots	 	    & $0$		 	& $0$		& $0$	\\
\vdots     &\vdots     &\vdots	& \vdots		 	& \vdots & \vdots   &	 \vdots	    & \vdots		& \vdots		& 	\vdots		& \vdots	\\
$0$        &$0$        &$0$	 & \dots		 	& $0$  & \dots    &	 $H^1$	 	& \dots	 	    & $H^d$		 	& $0$		& $0$	\\
$0$        &$0$        &$0$	 & \dots		 	& $0$  & \dots    &	 $H^0$	 	& \dots	 	    & $H^{d-1}$		 	& $H^d$			& $0$	\\
$0$        &$0$        &$0$	 & \dots		 	& $0$  & \dots    &	 $0$	 	& \dots	 	    & $H^{d-2}$		 	& $H^{d-1}$		& $0$	\\
\vdots     &\vdots     &\vdots	& \vdots		 	& \vdots & \vdots   &	 \vdots	    & \vdots		& \vdots		& 	\vdots		& \vdots	\\
$0$        &$0$        &$0$	 & \dots		 	& $0$  & \dots    &	 $0$	 	& \dots		 	& $H^1$		 	& $H^2$			& $0$	\\
$0$        &$0$        &$0$	 & \dots		 	& $0$  & \dots    &	 $0$	 	& \dots		 	& $H^0$		 	& $H^1$			& $0$	\\
$0$        &$0$        &$0$	 & \dots		 	& $0$  & \dots    &	 $0$	 	& \dots		 	& $0$		 	& $H^0$			& $H^0$	\\

\hline
\hline
$\Oo(-1,-1)$ & $\Gg^{m-1,n}$	 	 &$\Gg^{m,n-1}$ &\dots  &$\Gg^{0,n}$ &\dots & $\Gg^{n,0}$& \dots
&$\Gg^{0,1}$&$\Gg^{1,0}$ &$\Oo$\\
 \hline
\end{tabular}
\end{center}
By Remark \ref{sunto} in every column of the table at most one element is different to zero. We obtain

\begin{center}\begin{tabular}{|c|c|c|c|c|c|c|c|c|c|c|}
\hline
$\Oo(-m,-n)$ & $\Oo(-m+1,-n)$	 	 &$\Oo(-m,-n+1)$ &\dots  &$\Oo(0,-n)$ &\dots & $\Oo(-n,0)$& \dots
&$\Oo(0,-1)$&$\Oo(-1,0)$ &$\Oo$\\
\hline
\hline
$0$        &$0$    &$0$	 & \dots		 	& $0$  & \dots	    &	 $0$	& \dots		 	& $0$		 	& $0$			& $0$	\\
$0$        &$0$  &$0$	 & \dots		 	& $0$  & \dots     &	 $0$	 	& \dots		    & $0$		    & $0$			&$0$\\
\vdots     &\vdots     &\vdots	& \vdots		 	& \vdots & \vdots    &	 \vdots	    & \vdots		& \vdots		& 	\vdots		& \vdots	\\
$\alpha_0^{m,n}$        &$\alpha_0^{m-1,n}$  &$0$	 & \dots		 	& $0$  & \dots     &	 $0$	& \dots		 	& $0$		 	& $0$			 & $0$	\\
$0$        &$0$        &$\alpha_0^{m,n-1}$	 & \dots		 	& $0$  & \dots    &	 $0$	 	& \dots	 	    & $0$		 	& $0$		& $0$	\\
\vdots     &\vdots     &\vdots	& \vdots		 	& \vdots & \vdots    &	 \vdots	    & \vdots		& \vdots		& 	\vdots		& \vdots	\\
$0$        &$0$        &$0$	 & \dots		 	& $\alpha_0^{0,n}$  & \dots    &	 $0$	 	& \dots	 	    & $0$		 	& $0$		& $0$	\\
\vdots     &\vdots     &\vdots	& \vdots		 	& \vdots & \vdots   &	 \vdots	    & \vdots		& \vdots		& 	\vdots		& \vdots	\\
$0$        &$0$        &$0$	 & \dots		 	& $0$  & \dots    &	 $0$	 	& \dots	 	    & $0$		 	& $0$		& $0$	\\
$0$        &$0$        &$0$	 & \dots		 	& $0$  & \dots    &	 $\alpha_0^{m,0}$	 	& \dots	 	    & $0$		 	& $0$			& $0$	\\
$0$        &$0$        &$0$	 & \dots		 	& $0$  & \dots    &	 $0$	 	& \dots	 	    & $0$		 	& $0$		& $0$	\\
\vdots     &\vdots     &\vdots	& \vdots		 	& \vdots & \vdots   &	 \vdots	    & \vdots		& \vdots		& 	\vdots		& \vdots	\\
$0$        &$0$        &$0$	 & \dots		 	& $0$  & \dots    &	 $0$	 	& \dots		 	& $0$		 	& $0$			& $0$	\\
$0$        &$0$        &$0$	 & \dots		 	& $0$  & \dots    &	 $0$	 	& \dots		 	& $\alpha_0^{0,1}$		 	& $0$			& $0$	\\
$0$        &$0$        &$0$	 & \dots		 	& $0$  & \dots    &	 $0$	 	& \dots		 	& $0$		 	& $\alpha_0^{1,0}$			& $\alpha_0^{0,0}$	\\

\hline
\hline
$\Oo(-1,-1)$ & $\Gg^{m-1,n}$	 	 &$\Gg^{m,n-1}$ &\dots  &$\Gg^{0,n}$ &\dots & $\Gg^{n,0}$& \dots
&$\Gg^{0,1}$&$\Gg^{1,0}$ &$\Oo$\\
 \hline
\end{tabular}
\end{center}

 Using Beilinson's theorem in the strong form (as in Remark \ref{remb}) we retrieve the resolution (\ref{res}).

 Now we construct a Beilinson complex, quasi-isomorphic to $\Vv(-h)$, by calculating $H^{i+k_j}(\Vv\otimes \Ff_j)\otimes \Ee_j$ with  $i,j \in \{0, \ldots, d\}$ and by Remark \ref{sunto} in every column of the table at most one element is different to zero. So we get the following table:

\begin{center}\begin{tabular}{|c|c|c|c|c|c|c|c|c|c|c|}
\hline
$\Oo(-m,-n)$ & $\Oo(-m+1,-n)$	 	 &$\Oo(-m,-n+1)$ &\dots  &$\Oo(0,-n)$ &\dots & $\Oo(-n,0)$& \dots
&$\Oo(0,-1)$&$\Oo(-1,0)$ &$\Oo$\\
\hline
\hline
$0$        &$0$    &$0$	 & \dots		 	& $0$  & \dots	    &	 $0$	& \dots		 	& $0$		 	& $0$			& $0$	\\
$0$        &$0$  &$0$	 & \dots		 	& $0$  & \dots     &	 $0$	 	& \dots		    & $0$		    & $0$			&$0$\\
\vdots     &\vdots     &\vdots	& \vdots		 	& \vdots & \vdots    &	 \vdots	    & \vdots		& \vdots		& 	\vdots		& \vdots	\\
$\alpha_1^{m,n}$        &$\alpha_1^{m-1,n}$  &$0$	 & \dots		 	& $0$  & \dots     &	 $0$	& \dots		 	& $0$		 	& $0$			 & $0$	\\
$0$        &$0$        &$\alpha_1^{m,n-1}$	 & \dots		 	& $0$  & \dots    &	 $0$	 	& \dots	 	    & $0$		 	& $0$		& $0$	\\
\vdots     &\vdots     &\vdots	& \vdots		 	& \vdots & \vdots    &	 \vdots	    & \vdots		& \vdots		& 	\vdots		& \vdots	\\
$0$        &$0$        &$0$	 & \dots		 	& $\alpha_1^{0,n}$  & \dots    &	 $0$	 	& \dots	 	    & $0$		 	& $0$		& $0$	\\
\vdots     &\vdots     &\vdots	& \vdots		 	& \vdots & \vdots   &	 \vdots	    & \vdots		& \vdots		& 	\vdots		& \vdots	\\
$0$        &$0$        &$0$	 & \dots		 	& $0$  & \dots    &	 $\alpha_1^{m,0}$	 	& \dots	 	    & $0$		 	& $0$		& $0$	\\
$0$        &$0$        &$0$	 & \dots		 	& $0$  & \dots    &	 $0$	 	& \dots	 	    & $0$		 	& $0$			& $0$	\\
$0$        &$0$        &$0$	 & \dots		 	& $0$  & \dots    &	 $0$	 	& \dots	 	    & $0$		 	& $0$		& $0$	\\
\vdots     &\vdots     &\vdots	& \vdots		 	& \vdots & \vdots   &	 \vdots	    & \vdots		& \vdots		& 	\vdots		& \vdots	\\
$0$        &$0$        &$0$	 & \dots		 	& $0$  & \dots    &	 $0$	 	& \dots		 	& $\alpha_1^{0,1}$		 	& $0$			& $0$	\\
$0$        &$0$        &$0$	 & \dots		 	& $0$  & \dots    &	 $0$	 	& \dots		 	& $0$		 	& $\alpha_1^{1,0}$			& $0$	\\
$0$        &$0$        &$0$	 & \dots		 	& $0$  & \dots    &	 $0$	 	& \dots		 	& $0$		 	& $0$			& $0$	\\

\hline
\hline
$\Oo(-1,-1)$ & $\Gg^{m-1,n}$	 	 &$\Gg^{m,n-1}$ &\dots  &$\Gg^{0,n}$ &\dots & $\Gg^{n,0}$& \dots
&$\Gg^{0,1}$&$\Gg^{1,0}$ &$\Oo$\\
 \hline
\end{tabular}
\end{center}

 So we get the  resolution (\ref{res2}).

 Now we construct a Beilinson complex, quasi-isomorphic to $\Vv(-dh)$, by calculating $H^{i+k_j}(\Vv\otimes \Ff_j)\otimes \Ee_j$ with  $i,j \in \{0, \ldots, d\}$ and by Remark \ref{sunto} in every column of the table at most one element is different to zero. So we get the following table:

\begin{center}\begin{tabular}{|c|c|c|c|c|c|c|c|c|c|c|}
\hline
$\Oo(-m,-n)$ & $\Oo(-m+1,-n)$	 	 &$\Oo(-m,-n+1)$ &\dots  &$\Oo(0,-n)$ &\dots & $\Oo(-n,0)$& \dots
&$\Oo(0,-1)$&$\Oo(-1,0)$ &$\Oo$\\
\hline
\hline
$\alpha_d^{m,n}$        &$\alpha_d^{m-1,n}$    &$0$	 & \dots		 	& $0$  & \dots	    &	 $0$	& \dots		 	& $0$		 	& $0$			 & $0$	\\
$0$        &$0$  &$\alpha_d^{m,n-1}$	 & \dots		 	& $0$  & \dots     &	 $0$	 	& \dots		    & $0$		    & $0$			&$0$\\
\vdots     &\vdots     &\vdots	& \vdots		 	& \vdots & \vdots    &	 \vdots	    & \vdots		& \vdots		& 	\vdots		& \vdots	\\
$0$        &$0$  &$0$	 & \dots		 	& $0$  & \dots     &	 $0$	& \dots		 	& $0$		 	& $0$			 & $0$	\\
$0$        &$0$        &$0$	 & \dots		 	& $0$  & \dots    &	 $0$	 	& \dots	 	    & $0$		 	& $0$		& $0$	\\
\vdots     &\vdots     &\vdots	& \vdots		 	& \vdots & \vdots    &	 \vdots	    & \vdots		& \vdots		& 	\vdots		& \vdots	\\
$0$        &$0$        &$0$	 & \dots		 	& $\alpha_d^{0,n}$  & \dots    &	 $0$	 	& \dots	 	    & $0$		 	& $0$		& $0$	\\
\vdots     &\vdots     &\vdots	& \vdots		 	& \vdots & \vdots   &	 \vdots	    & \vdots		& \vdots		& 	\vdots		& \vdots	\\
$0$        &$0$        &$0$	 & \dots		 	& $0$  & \dots    &	 $\alpha_d^{m,0}$	 	& \dots	 	    & $0$		 	& $0$		& $0$	\\
$0$        &$0$        &$0$	 & \dots		 	& $0$  & \dots    &	 $0$	 	& \dots	 	    & $\alpha_d^{0,1}$		 	& $0$			& $0$	\\
$0$        &$0$        &$0$	 & \dots		 	& $0$  & \dots    &	 $0$	 	& \dots	 	    & $0$		 	& $\alpha_d^{1,0}$	&$0$	 \\
\vdots     &\vdots     &\vdots	& \vdots		 	& \vdots & \vdots   &	 \vdots	    & \vdots		& \vdots		& 	\vdots		& \vdots	\\
$0$        &$0$        &$0$	 & \dots		 	& $0$  & \dots    &	 $0$	 	& \dots		 	& $0$		 	& $0$			& $0$	\\
$0$        &$0$        &$0$	 & \dots		 	& $0$  & \dots    &	 $0$	 	& \dots		 	& $0$		 	& $0$			& $0$	\\
$0$        &$0$        &$0$	 & \dots		 	& $0$  & \dots    &	 $0$	 	& \dots		 	& $0$		 	& $0$			& $0$	\\

\hline
\hline
$\Oo(-1,-1)$ & $\Gg^{m-1,n}$	 	 &$\Gg^{m,n-1}$ &\dots  &$\Gg^{0,n}$ &\dots & $\Gg^{n,0}$& \dots
&$\Gg^{0,1}$&$\Gg^{1,0}$ &$\Oo$\\
 \hline
\end{tabular}
\end{center}

 So we get the  resolution (\ref{res3}).

 Now we construct a Beilinson complex, quasi-isomorphic to $\Vv(-(d-1)h)$, by calculating $H^{i+k_j}(\Vv\otimes \Ff_j)\otimes \Ee_j$ with  $i,j \in \{0, \ldots, d\}$ and by Remark \ref{sunto} in every column of the table at most one element is different to zero. So we get the following table:

\begin{center}\begin{tabular}{|c|c|c|c|c|c|c|c|c|c|c|}
\hline
$\Oo(-m,-n)$ & $\Oo(-m+1,-n)$	 	 &$\Oo(-m,-n+1)$ &\dots  &$\Oo(0,-n)$ &\dots & $\Oo(-n,0)$& \dots
&$\Oo(0,-1)$&$\Oo(-1,0)$ &$\Oo$\\
\hline
\hline
$0$        &$0$    &$0$	 & \dots		 	& $0$  & \dots	    &	 $0$	& \dots		 	& $0$		 	& $0$			 & $0$	\\
$\alpha_{d-1}^{m,n}$        &$\alpha_{d-1}^{m-1,n}$  &$0$	 & \dots		 	& $0$  & \dots     &	 $0$	 	& \dots		    & $0$		    & $0$			&$0$\\
$0$        &$0$  &$\alpha_{d-1}^{m-1,n}$	 & \dots		 	& $0$  & \dots     &	 $0$	& \dots		 	& $0$		 	& $0$			 & $0$	 \\
\vdots     &\vdots     &\vdots	& \vdots		 	& \vdots & \vdots    &	 \vdots	    & \vdots		& \vdots		& 	\vdots		& \vdots	\\
$0$        &$0$  &$0$	 & \dots		 	& $0$  & \dots     &	 $0$	& \dots		 	& $0$		 	& $0$			 & $0$	\\
\vdots     &\vdots     &\vdots	& \vdots		 	& \vdots & \vdots    &	 \vdots	    & \vdots		& \vdots		& 	\vdots		& \vdots	\\
$0$        &$0$        &$0$	 & \dots		 	& $\alpha_{d-1}^{0,n}$  & \dots    &	 $0$	 	& \dots	 	    & $0$		 	& $0$		& $0$	 \\
\vdots     &\vdots     &\vdots	& \vdots		 	& \vdots & \vdots   &	 \vdots	    & \vdots		& \vdots		& 	\vdots		& \vdots	\\
$0$        &$0$        &$0$	 & \dots		 	& $0$  & \dots    &	 $\alpha_{d-1}^{m,0}$	 	& \dots	 	    & $0$		 	& $0$		& $0$	\\
$0$        &$0$        &$0$	 & \dots		 	& $0$  & \dots    &	 $0$	 	& \dots	 	    & $0$		 	& $0$			& $0$	\\
$0$        &$0$        &$0$	 & \dots		 	& $0$  & \dots    &	 $0$	 	& \dots	 	    & $\alpha_{d-1}^{1,0}$		 	& $0$	&$0$	 \\
$0$        &$0$        &$0$	 & \dots		 	& $0$  & \dots    &	 $0$	 	& \dots		 	& $0$		 	& $\alpha_{d-1}^{0,1}$			& $0$	 \\
\vdots     &\vdots     &\vdots	& \vdots		 	& \vdots & \vdots   &	 \vdots	    & \vdots		& \vdots		& 	\vdots		& \vdots	\\
$0$        &$0$        &$0$	 & \dots		 	& $0$  & \dots    &	 $0$	 	& \dots		 	& $0$		 	& $0$			& $0$	\\

\hline
\hline
$\Oo(-1,-1)$ & $\Gg^{m-1,n}$	 	 &$\Gg^{m,n-1}$ &\dots  &$\Gg^{0,n}$ &\dots & $\Gg^{n,0}$& \dots
&$\Gg^{0,1}$&$\Gg^{1,0}$ &$\Oo$\\
 \hline
\end{tabular}
\end{center}

 So we get the  monad (\ref{res4}).

 Finally we construct a Beilinson complex, quasi-isomorphic to $\Vv(-qh)$ with $1<q<d-1$, by calculating $H^{i+k_j}(\Vv\otimes \Ff_j)\otimes \Ee_j$ with  $i,j \in \{0, \ldots, d\}$ and by Remark \ref{sunto} in every column of the table at most one element is different to zero. So we get the following table:

\begin{center}\begin{tabular}{|c|c|c|c|c|c|c|c|c|c|c|}
\hline
$\Oo(-m,-n)$ & $\Oo(-m+1,-n)$	 	 &$\Oo(-m,-n+1)$ &\dots  &$\Oo(0,-n)$ &\dots & $\Oo(-n,0)$& \dots
&$\Oo(0,-1)$&$\Oo(-1,0)$ &$\Oo$\\
\hline
\hline
$0$        &$0$    &$0$	 & \dots		 	& $0$  & \dots	    &	 $0$	& \dots		 	& $0$		 	& $0$			 & $0$	\\
\vdots     &\vdots     &\vdots	& \vdots		 	& \vdots & \vdots    &	 \vdots	    & \vdots		& \vdots		& 	\vdots		& \vdots	\\
$\alpha_{q}^{m,n}$        &$\alpha_{q}^{m-1,n}$  &$0$	 & \dots		 	& $0$  & \dots     &	 $0$	 	& \dots		    & $0$		    & $0$			 &$0$\\
$0$        &$0$  &$\alpha_{q}^{m-1,n}$	 & \dots		 	& $0$  & \dots     &	 $0$	& \dots		 	& $0$		 	& $0$			 & $0$	 \\
\vdots     &\vdots     &\vdots	& \vdots		 	& \vdots & \vdots    &	 \vdots	    & \vdots		& \vdots		& 	\vdots		& \vdots	\\
$0$        &$0$  &$0$	 & \dots		 	& $0$  & \dots     &	 $0$	& \dots		 	& $0$		 	& $0$			 & $0$	\\
\vdots     &\vdots     &\vdots	& \vdots		 	& \vdots & \vdots    &	 \vdots	    & \vdots		& \vdots		& 	\vdots		& \vdots	\\
$0$        &$0$        &$0$	 & \dots		 	& $\alpha_{q}^{0,n}$  & \dots    &	 $0$	 	& \dots	 	    & $0$		 	& $0$		& $0$	 \\
\vdots     &\vdots     &\vdots	& \vdots		 	& \vdots & \vdots   &	 \vdots	    & \vdots		& \vdots		& 	\vdots		& \vdots	\\
$0$        &$0$        &$0$	 & \dots		 	& $0$  & \dots    &	 $\alpha_{q}^{m,0}$	 	& \dots	 	    & $0$		 	& $0$		& $0$	\\
\vdots     &\vdots     &\vdots	& \vdots		 	& \vdots & \vdots    &	 \vdots	    & \vdots		& \vdots		& 	\vdots		& \vdots	\\
$0$        &$0$        &$0$	 & \dots		 	& $0$  & \dots    &	 $0$	 	& \dots	 	    & $\alpha_{q}^{1,0}$		 	& $0$	&$0$	 \\
$0$        &$0$        &$0$	 & \dots		 	& $0$  & \dots    &	 $0$	 	& \dots		 	& $0$		 	& $\alpha_{q}^{0,1}$			& $0$	 \\
\vdots     &\vdots     &\vdots	& \vdots		 	& \vdots & \vdots   &	 \vdots	    & \vdots		& \vdots		& 	\vdots		& \vdots	\\
$0$        &$0$        &$0$	 & \dots		 	& $0$  & \dots    &	 $0$	 	& \dots		 	& $0$		 	& $0$			& $0$	\\

\hline
\hline
$\Oo(-1,-1)$ & $\Gg^{m-1,n}$	 	 &$\Gg^{m,n-1}$ &\dots  &$\Gg^{0,n}$ &\dots & $\Gg^{n,0}$& \dots
&$\Gg^{0,1}$&$\Gg^{1,0}$ &$\Oo$\\
 \hline
\end{tabular}
\end{center}

 So we get the  monad (\ref{res5}).
\end{proof}
\begin{remark}Notice that if we tensor (\ref{res2}), (\ref{res4}), (\ref{res5}) and (\ref{res3}) by $\Oo(qh)$ with $q=1,\dots ,d$  and $\Oo(dh)$ we get resolutions or monads related to $\Vv$ that are all different from each other and from (\ref{res}).
\begin{enumerate}
\item From (\ref{res2}) we deduce that we must have $\alpha_1^{0,1}\not=0$ or $\alpha_1^{1,0}\not=0$.
\item From (\ref{res4}) and (\ref{res5}) we deduce that for any $q=2,\dots,d-1$ it is not possible to have $\alpha_q^{0,q}=\dots =\alpha_q^{q,0}=0$.
    \item If $\Sigma$ is a Segre Variety namely if $k_1=k_2=1$ we get that $\alpha_{i}^{m,n}=h^i(\Vv((-i-1)h))=0$ for any $i<d$ so (\ref{res}), (\ref{res2}), (\ref{res4}), (\ref{res5}) can be simplified.\\
        In particular (\ref{res4}) becomes
        $$0\to\Vv(-(d-1)h)\to\Oo(-m,-n+1)^{\oplus \alpha_{d-1}^{m,n-1}}\oplus\Oo(-m+1,-n)^{\oplus \alpha_{d-1}^{m-1,n}}\to$$ $$\to\oplus_{a+b=d-2}\Oo(-a,-b)^{\oplus \alpha_{d-1}^{a,b}}\to\dots\to\Oo(-1,0)^{\oplus \alpha_{d-1}^{1,0}}\oplus\Oo(0,-1)^{\oplus \alpha_{d-1}^{0,1}}\to  0,$$
    \end{enumerate}
\end{remark}

\section{Families of Ulrich bundles on $\Sigma$}\label{sec4}

We start this section with an example with families of Ulrich bundles supported on every $\Sigma$ obtained as pullbacks from the Veronese varieties $(\PP^m,\Oo_{\PP^m}(k_1))$ and $(\PP^n,\Oo_{\PP^n}(k_2))$:
\begin{example} Let $\Ee$ be an Ulrich bundle on $(\PP^m,\Oo_{\PP^m}(k_1))$ and $\Ff$ be an Ulrich bundle on $(\PP^n,\Oo_{\PP^n}(k_2))$, then $\Ee(nk_1)\boxtimes\Ff$, $\Ee\boxtimes\Ff(mk_2)$ are Ulrich bundles on $\Sigma$. In fact, for $t=1,\dots ,m+n$, $$\Ee(nk_1)\boxtimes\Ff\otimes\Oo(-th)$$ is acyclic since $\Ff(-tk_2)$ is acyclic in $(\PP^n,\Oo_{\PP^n}(k_2))$ for $t=1,\dots ,n$ and $\Ee((n-t)k_2)$ is acyclic in $(\PP^m,\Oo_{\PP^m}(k_1))$ for $t=n+1,\dots ,n+m$.
Similarly for $t=1,\dots ,m+n$, $$\Ee\boxtimes\Ff(mk_2)\otimes\Oo(-th)$$ is acyclic since $\Ff(-tk_2)$ is acyclic in $(\PP^n,\Oo_{\PP^n}(k_2))$ for $t=m+1,\dots ,m+n$ and $\Ee(-tk_2)$ is acyclic in $(\PP^m,\Oo_{\PP^m}(k_1))$ for $t=1,\dots ,m$.

In particular if $k_2=1$ then $\Ee(nk_1)\boxtimes\Oo_{\PP^n}$ and $\Ee\boxtimes\Oo_{\PP^n}(nk_2)$ are Ulrich bundles on $\Sigma$.
\end{example}

\begin{lemma}\label{ulb}
$\Sigma$ supports Ulrich line bundles if and only if we are in one of the following cases:
\begin{itemize}
\item $\Sigma$ is a Segre variety, namely $k_1=k_2=1$. In this case the only Ulrich line bundles are $\Oo(n,0),\Oo(0,m)$.
\item  $n=1$, $m>1$ and  $k_1=1$.  In this case the only Ulrich line bundles are  $\Oo(1,k_2-1)$ and  $\Oo(0,(m+1)k_2-1)$.
\item $m=n=1$. In this case the only Ulrich line bundles are $\Oo(k_1-1,2k_2-1),\Oo(2k_1-1,k_2-1)$.
\end{itemize}

\begin{remark}When $m=1$, $n>1$,  $k_2=1$, $\Sigma$ is a rational normal scroll of dimension two. So a complete classification of Ulrich bundles of every rank is given in \cite{FM}.
\end{remark}
\end{lemma}

\begin{proof}
$\Oo(a,b)$ is Ulrich if and only if $h^i(\Oo(a-tk_1,b-tk_2))=0$ for any integer $i$ and for $t=1,\dots ,m+n$. Recall that $h^i(\PP^m,\Oo_\PP^m(-s))=0$ for any $i$ if and only if $s=1,\dots ,m$ and $h^i(\PP^n,\Oo_\PP^n(-s))=0$ for any $i$ if and only if $s=1,\dots ,n$ so if $m>1$, since $h^i(\Oo(a-tk_1,b-tk_2))=\sum^{i_1+i_2=i}h^{i_1}(\PP^m,\Oo_\PP^m(a-tk_1))h^{i_2}(\PP^n,\Oo_\PP^2(b-tk_2))$, in order to have $\Oo(a,b)$ acyclic for $m+n$ consecutive twist we must have $k_1=k_2=1$. Moreover we get $a=0$ or $b=0$ hence the only Ulrich line bundles are $\Oo(n,0),\Oo(0,m)$.

If $n=1$ for any $k_2>0$ $\Oo_{\PP^1}(-1+k_2t)$ is acyclic for $t=0$ so when $m>1$ and $k_1>1$ there are not Ulrich line bundle and when $k_1=1$ the only Ulrich line bundles are $\Oo(1,k_2-1)$ and  $\Oo(0,(m+1)k_2-1)$.

If $m=n=1$ for any $k_1,k_2>0$ the only Ulrich line bundles are $\Oo(k_1-1,2k_2-1),\Oo(2k_1-1,k_2-1)$ (see \cite{ACM} and \cite{A}).

\end{proof}

\begin{lemma}\label{ulOm}
Let $a=1,\dots ,m, b=1,\dots, n$. $\Omega^a(l)\boxtimes\Omega^b(s)$ is Ulrich on $\Sigma$  in one of the following cases:
\begin{itemize}
\item[(1)] $\Sigma$ is a Segre variety, namely $k_1=k_2=1$, we have only $\Omega^1(n+1)\boxtimes\Omega^{n-1}(n),\Omega^{m-1}(m)\boxtimes\Omega^1(m+1),\Oo(n,0),\Oo(0,m)$.
    \item [(2)] $n=2$ $k_1=1$ and $k_2=2$, we have only  $\Omega^1(3)\boxtimes\Omega^1(7), \Omega^1(7)\boxtimes\Omega^1(3)$
    \item [(3)] $m=n=2$ and $k_1=k_2=2$, we have only  $\Omega^1(3)\boxtimes\Oo(2)$, $\Oo(2)\otimes\Omega^1(3)$
\item [(4)]  $n=1$, $m>1$ and  $k_1=1$. We have that $\Oo(k_1-1)\boxtimes\Omega^b(1)$ $\Omega^a(1)\boxtimes\Oo(k_2-1)$ is Ulrich for any $a=1,\dots , m$.
\item [(5)] $n=1$ and $m=k_1=2$, we have only  $\Oo(k_1-1)\otimes\Omega^1(3), \Oo(3k_1-1)\otimes\Omega^1(7)$ $\Omega^1(3)\boxtimes\Oo(k_2-1), \Omega^1(7)\boxtimes\Oo(3k_2-1)$
\item [(6)] $m=n=1$ we get $\Oo(k_1-1,2k_2-1),\Oo(2k_1-1,k_2-1)$.
\end{itemize}

\end{lemma}

\begin{proof}
$\Omega^a(l)\boxtimes\Omega^b(s)$ is Ulrich if and only if $h^i(\Omega^a(l-tk_1)\boxtimes\Omega^b(s-tk_2))=0$ for any integer $i$ and for $t=1,\dots , m+n$. Recall that $h^i(\PP^m,\Omega^a_{\PP^m}(l))=0$ for any $i$ if and only if $l=a-1,\dots ,a-m$ except $l=0$ and $h^i(\PP^n,\Omega^b_{\PP^n}(s))=0$ for any $i$ if and only if $s=b-1,\dots ,b-n$ except $s=0$ so if $m>1$, since $h^i(h^i(\Omega^a(l-tk_1)\boxtimes\Omega^b(s-tk_2)))=\sum^{i_1+i_2=i}h^{i_1}(\PP^m,\Omega^a_{\PP^n}(l-tk_1))h^{i_2}(\PP^n,\Omega^b_{\PP^n}(s-tk_2))$, in order to have $\Omega^a(l)\boxtimes\Omega^b(s)$ acyclic for $m+n$ consecutive twist we must have $k_1=k_2=1$ or $m=k_1=2$ and $k_2=1$ or $m=n=2$ and $k_1=k_2=2$.\\
 Let $m>2$ and $k_1=k_2=1$. If $a=n$ $\Omega_{\PP^m}^a(l)$ is a line bundle $\Omega_{\PP^n}^b(s)$ must have $n$ consecutive acyclic twist so in must be also a line bundle. A similar argument can be used when $b=m$. Since these cases have been considered in Lemma \ref{ulb} we may assume $a=1,\dots ,m-1, b=1,\dots, n-1$.  Moreover we must have $l=a+1$ or $s=b+1$. Let assume $a=b+1$, In order to have $m+n$ acyclic twist we must have $h^{m}(\PP^m,\Omega^a_{\PP^m}(l-(m+n))=0$ so $l-m-n\geq a-m$, then $l\geq a+n$.  Since $h^a(\PP^m,\Omega^a_{\PP^m})\not=0$ we must have $h^n(\PP^n,\Omega^b_{\PP^n}(b+1-t))=0$, hence $b+1-t\geq b-n$ , then $l\leq n+1$. We may conclude that $n+1\leq l\leq a+n$ hence $a=1$ and $l=n+1$. Finally, since $h^b(\PP^n,\Omega^b_{\PP^n})\not=0$, we must have $h^0(\PP^m,\Omega^1_{\PP^m}(n+1-b-1))=0$, which implies $n-b\leq 1$ so $b\geq n-1$. We have obtained $\Omega^1(n+1)\boxtimes\Omega^{n-1}(n)$.\\
Similarly if we assume $l=a+1$ we obtain $\Omega^{m-1}(m)\boxtimes\Omega^1(m+1)$. So we have proved $(1)$.\\

If $n\leq 2$ by a simple case by case analysis we get $(2), (3)$.\\

If $n=1$ for any $k_2>0$, $\Oo_{\PP^1}(-1+k_2t)$ is acyclic for $t=0$ so when $m>1$ and $k_1>1$ there are not Ulrich bundle $\Omega^{a}(s)\boxtimes\Oo(l)$. When $k_1=1$ we must have $h^i(\PP^m,\Omega^b_{\PP^m}(s))=0$ for any $i$ and for $m+1$ consecutive twist except $s=-1$ so $\Omega^a_{\PP^m}(s)$ can be $\Omega^a(1)$ for any $a=1,\dots ,m$. So we get $(4)$. If $m=k_1=2$ we have the two cases described in $(5)$.\\

If $m=n=1$ by Lemma \ref{ulb} we obtain also $(6)$.

\end{proof}

\begin{remark} For  cohomological characterizations of $\Omega^a(l)\boxtimes\Omega^b(s)$ see \cite{MM}.
\end{remark}

\begin{proposition}\label{culb}
Let $\Vv$ be an indecomposable Ulrich bundle on $\Sigma$. Let $m>1$, then
\begin{itemize}

\item [(a)]  when $n>1$, if $\alpha_1^{0,2}=\alpha_1^{1,1}=0$ we must have  $k_2=1$ and $\Vv=\Oo(k_1-1,1)$,
\item [(b)] when $n=1$, if $h^1(\Vv\otimes\Oo(-2k_1,-1-k_2))=0$  we get $\Vv=\Oo(2k_1-1,k_2-1)$ and if $h^1(\Vv\otimes\Oo(-1-k_1,-2k_2))=0$  we get $\Vv=\Oo(k_1-1,2k_2-1)$.
\end{itemize}

\end{proposition}

\begin{proof}
$(a)$ Let us consider the resolution (\ref{res2}). Since $\alpha_1^{0,2}=\alpha_1^{1,1}=0$ we get $\Vv(-h)=\Oo(-1,0)^{\oplus \alpha_1^{1,0}}\oplus\Oo(0,-1)^{\oplus \alpha_1^{0,1}}$. So by Lemma \ref{ulb} we obtain that $k_2=1$ and $\Vv(-h)=\Oo(-1,0)$.

$(b)$ after $k_1-1$ right mutations the full exceptional collection $\{\Oo(-1,-1),\Oo(0,-1),\Oo(-1,0),\Oo\}$ becomes $\{\Oo(k_1-2,-1),\Oo(k_1-1,-1),\Oo(-1,0),\Oo\}$. So we obtain the following Beilinson table for $\Vv(-h)$:
\begin{center}\begin{tabular}{|c|c|c|c|}
                \hline
                $\Oo(k_1-2,-1)$ & $\Oo(k_1-1,-1)$ &$\Oo(-1,0)$ &$\Oo$ \\
                \hline
                \hline
                $0$ & $0$ & $0$ & $0$ \\
                $a$ & $b$ & $0$ & $0$ \\
                $0$ & $0$ & $c$ & $0$ \\
                $0$ & $0$ & $0$ & $0$ \\
                \hline
                \hline
                $\Oo(-k_1,-1)$ & $\Oo(-k_1+1,-1)$ &$\Oo(-1,0)$ &$\Oo$ \\
                \hline
              \end{tabular}\end{center}
              where $a=h^1(\Vv(-h)\otimes\Oo(-k_1,-1))$, $b=h^1(\Vv(-h)\otimes\Oo(-k_1+1,-1))$, $c=h^1(\Vv(-h)\otimes\Oo(-1,0))$. Since $a=h^1(\Vv\otimes\Oo(-2k_1,-1-k_2))=0$ we obtain $\Vv(-h)=\Oo(k_1-1,-1)^{\oplus b}\oplus\Oo(-1,0)^{\oplus c}$. Finally by Lemma \ref{ulb} we may conclude that $\Vv=\Oo(2k_1-1,k_2-1)$. Similarly if $h^1(\Vv\otimes\Oo(-1-k_1,-2k_2))=0$  we get $\Vv=\Oo(k_1-1,2k_2-1)$.
\end{proof}

\begin{theorem}\label{vlb}
Let $\Vv$ be an indecomposable Ulrich bundle on $\Sigma$. Let $n>1$, $a=0,\dots, m$, $b=0,\dots , n$. Let $\alpha_{m}^{a,b}=0$ if $a+b=m-1$ or $a+b=m+1$. Them we must have $k_1=k_2=1$ and $\Vv=\Oo(0,m)$ if $m\not=n$ or $\Vv=\Oo(n,0), \Vv=\Oo(0,m)$ if $m=n$.

\end{theorem}
\begin{proof}

Let us consider the monad (\ref{res5}) for $q=m$:
$$0\to\Bb_1\to\oplus_{a+b=m}\Oo(-a,-b)^{\oplus \alpha_{q}^{a,b}}\to\Bb_2\to 0.$$
Since $\alpha_{m}^{a,b}=0$ if $a+b=m-1$ we get $\Bb_1=0$ and since $\alpha_{m}^{a,b}=0$ if  $a+b=m+1$ we get $\Bb_2=0$. So
$$\Vv(-mh)=\oplus_{a+b=m}\Oo(-a,-b)^{\oplus \alpha_{q}^{a,b}}$$ and, since $\Vv$ is indecomposable, by Lemma \ref{ulb} we get the claim.

\end{proof}

\begin{proposition}\label{culq}
Let $\Vv$ be an Ulrich bundle on $\Sigma$. Let $k_2=1$,  if $\alpha_n^{a,b}=0$ when $a+b=n+1$ or $a>1$ and $a+b=n$ we get $\Vv=\Ee(nk_1)\boxtimes\Oo$, where $\Ee$ be an Ulrich bundle on $(\PP^m,\Oo_{\PP^m}(k_1))$.

\end{proposition}

\begin{proof}
$(a)$ Let us consider the monad for $\Vv(-nh)$ (\ref{res5}):
$$0\to\Bb_1\to\oplus_{a+b=n}\Oo(-a,-b)^{\oplus \alpha_{n}^{a,b}}\to\Bb_2\to 0$$
where $\Bb_1$ is defined by the exact sequence
$$0\to\Oo(-m,-n)^{\oplus \alpha_{n}^{m,n}}\to\dots\to \oplus_{a+b=n-1}\Oo(-a,-b)^{\oplus \alpha_{n}^{a,b}}\to\Bb_1\to 0$$
and $\Bb_2$ is defined by the exact sequence
$$0\to\Bb_2\to \oplus_{a+b=n+1}\Oo(-a,-b)^{\oplus \alpha_{n}^{a,b}}\to\dots\to\to\Oo(-1,0)^{\oplus \alpha_{n}^{1,0}}\oplus\Oo(0,-1)^{\oplus \alpha_{n}^{0,1}}\to 0.$$ Since $\alpha_n^{a,b}=0$ if $a+b=n+1$ we get $\Bb_2=0$.\\
Since $\alpha_n^{a,b}=0$ if $a>1$ and $a+b=n$ we get

$$0\to\Oo(-m,-n)^{\oplus \alpha_{n}^{m,n}}\to\dots\to \oplus_{a+b=n-1}\Oo(-a,-b)^{\oplus \alpha_{n}^{a,b}}\to \Oo(0,-n)^{\oplus \alpha_{n}^{0,n}}\to\Vv(-nh)\to 0.$$
Now, since when $b<n$ there no maps from $\Oo(-a,-b)$ to $\Oo(0,-n)$, we must have $ \alpha_{n}^{a,b}=0$ when $b<n$ and $a+b=n+1,\dots ,m+n-1$, so we obtain

$$0\to\Oo(-m,-n)^{\oplus \alpha_{n}^{m,n}}\to\dots\to \Oo(-1,-n)^{\oplus \alpha_{n}^{1,n}}\to \Oo(0,-n)^{\oplus \alpha_{n}^{0,n}}\to\Vv(-nh)\to 0$$
and we may conclude that $\Vv=\Ee(nk_1)\boxtimes\Oo$, where $\Ee$ is an Ulrich bundle on $(\PP^m,\Oo_{\PP^m}(k_1))$.
\end{proof}

\section{The general case}\label{sec5}

Theorem \ref{volon} can be generalized to the general case as follow:

\begin{theorem}\label{volong}
Let $\Vv$ be an Ulrich bundle on $\Sigma$. Let $a_i=0,\dots, n_i$. Then  for $0\leq q\leq d$ $\Vv(-qh)$ is the homology of the following monad

\begin{equation}\label{res5g}
0\to\Bb_1\to\oplus_{a_1+\dots +a_s=q}\Oo(-a_1,\dots ,-a_s)^{\oplus \alpha_{q}^{a_1,\dots ,a_s}}\to\Bb_2\to 0\end{equation}
where $\Bb_1$ is defined by the exact sequence
$$0\to\Oo(-n_1,\dots ,-n_s)^{\oplus \alpha_{q}^{n_1,\dots ,n_s}}\to\dots\to \oplus_{a_1+\dots +a_s=q+1}\Oo(-a_1,\dots ,-a_s)^{\oplus \alpha_{q}^{a_1,\dots ,a_s}}\to\Bb_1\to 0$$
and $\Bb_2$ is defined by the exact sequence
$$0\to\Bb_2\to \oplus_{a_1+\dots +a_s=q-1}\Oo(-a_1,\dots ,-a_s)^{\oplus \alpha_{q}^{a_1+\dots +a_s}}\to\dots\to\oplus_{a_1+\dots +a_s=1}\Oo(-a_1,\dots ,-a_s)^{\oplus \alpha_{q}^{a_1+\dots +a_s}}\to 0$$

\end{theorem}
\begin{proof}

 Let us consider on every $\PP^{n_i}$ the full exceptional collection $$\{\Oo_{\PP^{n_i}}(-n_i),\dots , \Oo_{\PP^{n_i}}(-1), \Oo_{\PP^{n_i}}\}.$$
  We may obtain (see \cite{Orlov}) a full exceptional collection $\langle E_d \ldots, E_0\rangle$ given by bundles $\Oo(j_1, \dots, j_s)$ $(0\leq j_i\leq n_i)$ with suitable shift which is the generalization of (\ref{col}).

The associated  exceptional collection $\langle F_d, \ldots, F_0\rangle$ of Theorem \ref{use}, which is the generalization of (\ref{cold}), is given by bundles $\Gg^{a_1,\dots, a_s}=\Omega^{a_1}(a_1)\boxtimes\dots\boxtimes\Omega^{a_s}(a_s)$.

 We construct a Beilinson complex, quasi-isomorphic to $\Vv(-qh)$, by calculating $H^{i+k_j}(\Vv\otimes \Ff_j)\otimes \Ee_j$ with  $i,j \in \{0, \ldots, d\}$. By Remark \ref{suntog} in every column of the table at most one element is different to zero. So, arguing as in Theorem \ref{volon} we get   the  monad (\ref{res5g}).
\end{proof}
\begin{remark}For $q=0$, $q=1$, $q=d$ we obtain the resolutions. The resolutions and monads are all different from each other for any $q=0,\dots, d$.
\begin{enumerate}

\item From  (\ref{res5g}) we deduce that for any $q=0,\dots,d$ it is not possible to have $\alpha_q^{a_1,\dots, a_s}=0$ for every $a_1,\dots ,a_s$ such that  $0\leq a_i\leq n_i$ for every index $i$ and $a_1+\dots+a_s=q$.
    \item If $\Sigma$ is a Segre Variety namely if $k_1=\dots =k_s=1$ we get that $$\alpha_{i}^{n_1,\dots ,n_s}=h^i(\Vv((-i-1)h))=0$$ for any $i<d$ so the bundle $\Bb_1$ in (\ref{res5g}) can be simplified:
        $$0\to\oplus_{a_1+\dots +a_s=d-1}\Oo(-a_1,\dots ,-a_s)^{\oplus \alpha_{q}^{a_1,\dots ,a_s}}\to\dots\to \oplus_{a_1+\dots +a_s=q+1}\Oo(-a_1,\dots ,-a_s)^{\oplus \alpha_{q}^{a_1,\dots ,a_s}}\to\Bb_1\to 0.$$
        \item The results of section \ref{sec4} can be generalized on arbitrary Segre Veronese varieties.

    \end{enumerate}
\end{remark}


\providecommand{\bysame}{\leavevmode\hbox to3em{\hrulefill}\thinspace}
\providecommand{\MR}{\relax\ifhmode\unskip\space\fi MR }
\providecommand{\MRhref}[2]{%
  \href{http://www.ams.org/mathscinet-getitem?mr=#1}{#2}
}
\providecommand{\href}[2]{#2}

\end{document}